\documentclass[12pt]{amsart}
\usepackage{amssymb,amsfonts, amsmath}
 
\newcommand{\ssubsection}[1]{\vspace{3mm}\noindent{{\bf #1.} }}
\numberwithin{equation}{section}

\renewcommand{\geq}{\geqslant}
\renewcommand{\leq}{\leqslant}
\newcommand{\Osh}{{\mathcal O}}                        
\renewcommand{\H}{\mathrm{H}}                          
\newcommand{\K}{\mathrm{K}}                
\newcommand{\N}{\mathrm{N}}            
\newcommand{\Ish}{\mathcal{I}}
\newcommand{\spec}{\operatorname{Spec}}
\renewcommand{\emptyset}{\varnothing}
\newcommand{\KK}{\mathbf{K}}
\newcommand{\FF}{\mathbf{F}}
\newcommand{\NN}{\mathbb{N}} 
\newcommand{\PP}{\mathbb{P}} 
\newcommand{\QQ}{\mathbb{Q}} 
\newcommand{\RR}{\mathbb{R}} 

\newtheorem{theorem}{Theorem}[section]
\newtheorem{lemma}[theorem]{Lemma}

\newtheorem{conjecture}[theorem]{Conjecture}
\theoremstyle{definition}
\newtheorem{defn}[theorem]{Definition}
\newtheorem{remark}[theorem]{Remark}
\newtheorem{example}[theorem]{Example}

\begin{document}
\title[Arithmetic inequalities for points of bounded degree]{On arithmetic inequalities for points of bounded degree}
\author{Nathan Grieve}
\email{
nathan.m.grieve@gmail.com
}%

\begin{abstract} 
We study algebraic points of bounded degree on polarized projective varieties.  To do so, we refine further the filtration construction and Subspace Theorem approach, for the study of integral points, which has origins in the work of Corvaja-Zannier, Levin, Evertse and Autissier.  Our main result shows how a conjecture of H.~P.~Schlickewei implies  Second Main Arithmetic Schmidt's Subspace type inequalities for polarized projective varieties and points of bounded degree.  
\end{abstract}
\thanks{This article has been published in Research in Number Theory. The final published version is available at: https://doi.org/10.1007/s40993-020-00226-w.}  

\maketitle

\section{Introduction}

Our starting point here is \cite[Conjecture 5.1]{Schlickewei:2003} in which H.~P.~Schlickewei proposed certain arithmetic inequalities for algebraic points of projective $r$ space with fixed bounded degree.  Results in the direction 
of this conjecture have been obtained by Levin  \cite{Levin:2014}.  The main approach given there is based on the concept of subgeneral position for hyperplanes in projective space.  
They are quite different from those of Evertse, Ferretti and Schlickewei  \cite{Evertse:Schlickewei:2002},  \cite{Schlickewei:2003} and \cite{Evertse:Ferretti:2013}.  For example,  recall that the quantitative (absolute) form of Schmidt's Subspace Theorem is deduced as an application of its parametric generalization \cite{Evertse:Schlickewei:2002}, \cite{Evertse:Ferretti:2013}. 

In general, it remains an important and difficult question to develop systematic methods for obtaining a Second Main Theorem for points of bounded degree.  Such matters, from the point of view of Vojta's Main Conjecture, with discriminant term, are discussed in \cite[Appendix]{Levin:2014}.  Our techniques here build on the filtration methods of \cite{Corvaja:Zannier:2002}, \cite{Autissier:2011} and \cite{Ru:Vojta:2016}.  To place matters into their proper context, recall that the Main Conjecture  implies the $abc$ Conjecture \cite{Vojta:1998}.  

Returning to \cite[Conjecture 5.1]{Schlickewei:2003}, it is suggested that
the optimal constant that appears in these inequalities should be governed by a function which is linear in the number of polynomial variables and which is quadratic in the degree of the algebraic points.  Moreover, the exceptional set predicted by \cite[Conjecture 5.1]{Schlickewei:2003} is contained in a finite union of linear subspaces, each of which has field of definition with degree at most that of the given degree of the fields of definition of the algebraic points.  We adopt that view point here; see Theorem \ref{Schlickewei:linear:systems:points:bounded:degree}.

Building on the approach of \cite{Ru:Vojta:2016} and \cite{Grieve:Function:Fields}, for example, we reformulate \cite[Conjecture 5.1]{Schlickewei:2003} and state it using the language of linear series (see Theorem \ref{Schlickewei:linear:systems:points:bounded:degree}).  We then adapt the techniques of \cite{Ru:Vojta:2016} to show how it can be used to establish general arithmetic inequalities for algebraic points of fixed bounded degree. They are in the spirit of the classical Second Main (Schmidt's Subspace) Theorem.  We formulate this result as Theorem \ref{Arithmetic:General:Theorem:Points:Bounded:Degree} below.  
 
\begin{theorem}\label{Arithmetic:General:Theorem:Points:Bounded:Degree}
The main inequalities for algebraic points in projective $r$ space, with fixed  bounded degree $d$, as formulated in Conjecture \ref{Points:bounded:degree}, have the following consequences for a given  geometrically irreducible projective variety $X$ over a fixed base number field $\KK$.  Let $D_1,\dots,D_q$ be nonzero effective Cartier divisors on $X$ and defined over a fixed finite extension field $\FF / \KK$, with $\KK \subseteq \FF \subseteq \overline{\KK}$, for $\overline{\KK}$ a fixed algebraic closure of $\KK$.  Put 
$$D = D_1 + \dots + D_q \text{,}$$
 and assume that these divisors $D_i$, for $i = 1,\dots, q$, intersect properly over $\FF$.  
 
 Let $L$ be a big line bundle on $X$ and defined over $\KK$.  Fix a finite set of places $S \subseteq M_{\KK}$.   Then, for each $i = 1,\dots, q$, there exist positive constants $\gamma(d,L,D_i)$ so that for each given $\epsilon > 0$, the inequality 
\begin{align*}
\sum_{i=1}^q  \gamma(d,L,D_i)^{-1} m_S(x,D_i) & \leq \left(1 + \epsilon \right) h_{L}(x) + \mathrm{O}(1)
\end{align*}
holds true
for all algebraic points
$$x \in X(\overline{\KK}) \setminus \left( Z(\overline{\KK}) \bigcup \operatorname{Bs}(L)(\overline{\KK}) \bigcup \operatorname{Supp}(D)(\overline{\KK})\right)$$
with 
 $[\KK(x):\KK] \leq d \text{.}$  
 Here, 
 $$m_S(x,D_i) := \sum_{v \in S} \lambda_{D_i,v}(x)\text{,}$$ 
 for $i = 1,\dots,q$, is the proximity function of $D_i$ with respect to $S$, $\operatorname{Bs}(L)$ is the base locus of $L$, $\operatorname{Supp}(D)$ is the support of $D$ and $Z \subsetneq X$ is a proper Zariski closed subset which is contained in a finite union of \emph{linear sections} $\Lambda_1,\dots,\Lambda_h$ each of which has field of definition with degree at most equal to $d$ over $\KK$.
\end{theorem}

We prove Theorem \ref{Arithmetic:General:Theorem:Points:Bounded:Degree} in \S \ref{proof:Arithmetic:General:Theorem}.  We define the concept of \emph{linear section} in Definition \ref{linear:sections}.  Conjecture \ref{Points:bounded:degree} is stated in Section \ref{abs:values:Weil:functions}.  To help place it into its proper context, we mention that explicit effective results in the direction of this conjecture, for the case of quadratic points in the plane, for instance, are obtained in \cite[Theorem 1.5 and Section 6]{Levin:2014}.  Further, note that the case $d = 1$ recovers the statement of the Arithmetic General theorem from \cite[p. 961]{Ru:Vojta:2016}.  It also describes, in more explicit terms, the nature of the diophantine exceptional set.

\section{Preliminaries}\label{abs:values:Weil:functions} 
  
Let $\KK$ be a number field with algebraic closure $\overline{\KK}$ and set of places $M_{\KK}$.    Our conventions about absolute values are consistent with those of 
\cite[p.~11]{Bombieri:Gubler}.  For example, if $p$ is a prime number, then 
$$|p|_p = p^{-1}.$$  
More generally, if $v \in M_{\KK}$, then $v \mid p$, for some place $p \in M_{\QQ}$ and 
\begin{equation}\label{abs:value:convention}
|x|_v := \left| \N_{\KK_v / \QQ_p}(x)\right|_p^{\frac{1}{[\KK:\QQ]}}
\end{equation}
for $x \in \KK$.  With these conventions, the set $M_\KK$ satisfies the product formula with multiplicities equal to one.   

Let $\FF / \KK$ be a finite extension with $\KK \subseteq \FF \subseteq \overline{\KK}$. If $v \in M_{\KK}$, $w \in M_{\FF}$, $w \mid v$ and $x \in \KK$, then
$$
|x|_w = |x|_v^{[\FF_w : \KK_v] / [ \FF : \KK] } \text{.}
$$
If $v \in M_{\KK}$, then we denote by $|\cdot|_{v,\KK}$ a fixed  extension of $|\cdot|_v$ to $\overline{\KK}$.  When no confusion is likely, we also write $|\cdot|_v$ in place of $|\cdot|_{v,\KK}$.

If 
$$x = [x_0 : \dots : x_r] \in \PP^r_{\overline{\KK}}\left(\overline{\KK}\right)$$
 has field of definition $\KK(x)$, then fix a number field $\FF$ with $\KK(x) \subseteq \FF \subseteq \overline{\KK}$.  Put 
\begin{equation}\label{abs:height:defn}
h_{\Osh_{\PP^r}(1)}(x) := \sum_{w \in M_{\FF}} \max_j \log |x_j|_w \text{.}
\end{equation}
Then \eqref{abs:height:defn} is the \emph{absolute logarithmic height} of $x  \in \PP^r(\overline{\KK})$ with respect to the \emph{tautological line bundle} $\Osh_{\PP^r_{\overline{\KK}}}(1)$.  It is well-defined and, thus, independent of $\FF$ and the choice of homogeneous coordinate vector.

Let $L$ be an ample line bundle on a geometrically irreducible projective variety $X$.  Assume that $(X,L)$ is defined over $\KK$.  Fix $m \gg 0$ so that $mL$ is very ample and define
$$
h_L(\cdot) := \frac{1}{m} h_{mL}(\cdot)
$$
 for $h_{mL}(\cdot)$ the pullback of $h_{\Osh_{\PP^r}(1)}(\cdot)$, for $r := h^0(X,mL)-1$, with respect to the embedding $X \hookrightarrow \PP^r$ that is afforded by the complete linear system $|mL|$.  
 This is the \emph{logarithmic height function} of $X$ with respect to $L$.  It is well-defined, up to bounded functions, and has domain the set of $\overline{\KK}$-points of $X$.  
 
 For an arbitrary line bundle $L$ on $X$, and defined over $\KK$, write $L$ as a difference of ample line bundles $M$ and $N$. Then define $h_L(\cdot)$ by
 $$
 h_L(\cdot) = h_{M-N}(\cdot) := h_M(\cdot) - h_N(\cdot)  + \mathrm{O}(1).
 $$
 Again, $h_L(\cdot)$ is well-defined, up to bounded functions, and has domain $X(\overline{\KK})$.
 
 Turning to \emph{local Weil functions}, let $D$ be a Cartier divisor on $X$ and having field of definition $\FF$ a finite extension of $\KK$, with $\KK \subseteq \FF \subseteq \overline{\KK}$.  Write 
 $$D = \operatorname{div}(s)$$ 
 for 
 $$s=s_D$$  a global meromorphic section of the line bundle $\Osh_X(D)$.  Let $M$ and $N$ be line bundles that are, respectively, generated by global sections $s_0,\dots,s_k$ and $t_0,\dots,t_\ell$ and which have the property that
 $$
\Osh_X(D) \simeq M \otimes N^{-1}.
 $$
 
 The \emph{local Weil function} 
 $$\lambda_{s,v}(\cdot) = \lambda_{D,v}(\cdot)$$
 of $D$ with respect to a place $v \in M_{\KK}$ and with \emph{presentation} 
 $$\mathcal{D} = (s_D; M,s_0,\dots,s_k;N,t_0,\dots,t_\ell)$$ 
 has domain $X(\overline{\KK}) \setminus \operatorname{Supp}(D)(\overline{\KK})$, for $\operatorname{Supp}(D)$ the support of $D$.  It is defined by
$$
\lambda_{D,v}(x)
 = \max_i \min_j \log \left(\left| \frac{s_i}{t_j s}(x) \right|_{v,\KK} \right) \text{.}
$$
Such local Weil functions are well-defined up to bounded functions.

By \cite[Proposition 2.3.9]{Bombieri:Gubler}, each effective Cartier divisor admits a presentation for which the corresponding Weil function is nonnegative at all points of its domain.  A more general, birational, formalism of local Weil functions is described in \cite{Ru:Vojta:2016} and \cite{Grieve:Divisorial:Instab:Vojta}.  

Finally, fixing a finite set of places $S \subseteq M_\KK$, the \emph{proximity function} of $D$ with respect to $S$ is defined to be
 $$
 m_S(x,D) := \sum_{v \in S} \lambda_{D,v}(x) \text{.}
 $$
It has domain $X(\overline{\KK}) \setminus \operatorname{Supp}(D)(\overline{\KK})$ and is well-defined up to bounded functions.
  
 \begin{example}[Compare with {\cite[\S 2.7.7]{Bombieri:Gubler}}]
Given a linear form
$\ell(x) \in \FF[x_0,\dots,x_r]\text{,}$ 
defining a hyperplane $H$ in $\PP^r$, and having field of definition $\FF$, set
$$
|| \ell(x) ||_{v,\KK} := \frac{|\ell(x)|_{v,\KK} }{ \max_j |x_j|_{v, \KK } } =  \frac{|\ell(x)|_{v} }{ \max_j |x_j|_{v} } \text{.}
$$
This is the \emph{standard metric} of $\ell(x)$ with respect to $v$.  It determines the \emph{local Weil function} of $\ell(x)$ with respect to $v$.  Explicitly
\begin{equation}\label{standard:metric:local:weil}
\lambda_{\ell(x),v}(x) := - \log \left| \left| \ell(x) \right| \right|_{v,\KK}   = \log \left( \frac{ \max_j \left| x_j \right|_v }{ \left| \ell(x) \right|_v }\right) \text{,}
\end{equation}
for $x = [x_0:\dots:x_r] \in \PP^r(\overline{\KK})$ not contained in the hyperplane $H$.  This is the local Weil function for $H$ with respect to the presentation $\mathcal{H} = (\ell(x); \Osh_{\PP^r}(1);x_0,\dots,x_r;\Osh_{\PP^r},1)$.
\end{example}

For later use, the form of \cite[Conjecture 5.1]{Schlickewei:2003}, which we use  here, is formulated in the following way.

\begin{conjecture}[The main inequalities for points of bounded degree]\label{Points:bounded:degree}
Let $\KK$ be a number field with fixed algebraic closure $\overline{\KK}$.  Fix $S$ a finite set of places of $\KK$.
Let $r,d$ be positive natural numbers.  For each $v \in S$, fix linearly independent linear forms $\ell_{v,0}(x),\dots,\ell_{v,r}(x)$ in the polynomial ring $\overline{\KK}[x_0,\dots,x_r]$.  Then there exists a positive constant $\mathrm{c}(r,d) > 0$, which depends only on $r$ and $d$, which has the following property for each fixed $\delta > 0$.  If $Z \subsetneq \PP^r(\overline{\KK})$ is the set of all points 
$x = [x_0:\dots:x_r] \in \PP^r(\overline{\KK})$ 
which satisfy the conditions that
\begin{itemize}
\item{
$
\sum_{v \in S} \sum_{i=0}^r \lambda_{\ell_{v,i},v}(x) > (\mathrm{c}(r,d) + \delta)h_{\Osh_{\PP^r_{\overline{\KK}}}(1)}(x) + \mathrm{O}(1)
$; and }
\item{
$[\KK(x) : \KK] \leq d$, 
}
\end{itemize}
then there exist finitely many proper linear subspaces $\Lambda_1,\dots, \Lambda_h$ in $\PP^r_{\overline{\KK}}$, each having field of definition with degree at most $d$ over $\KK$, and such that $Z$ is contained in their union $\Lambda_1 \bigcup \hdots \bigcup \Lambda_h$.
\end{conjecture}

Note that when $d = 1$, Conjecture \ref{Points:bounded:degree} is a slight variant to the traditional statement of Schmidt's Subspace Theorem \cite[Section 7.2]{Bombieri:Gubler}.

\section{Schlickewei's inequalities expressed in terms of linear systems}

As is the case for the Second Main Theorem, for example in the form of \cite{Ru:Vojta:2016} or \cite{Grieve:2018:autissier}, it is useful to express the inequalities predicted by \cite[Conjecture 5.1]{Schlickewei:2003} in terms of linear systems.   Formulating such diophantine geometric  inequalities in this way is by now standard practice (see for instance \cite{Faltings:Wustholz}, \cite{Ru:Vojta:2016}).  Before doing this, see Theorem \ref{Schlickewei:linear:systems:points:bounded:degree} below, we define the following concept which is useful for expressing its conclusion.

\begin{defn}\label{linear:sections}
Let $L$ be a very ample line bundle on a geometrically irreducible projective variety $X$.  Assume that $(X,L)$ is defined over $\KK$.  Let $\FF / \KK$ be a finite extension, $\KK \subseteq \FF \subseteq \overline{\KK}$.  Put $V := \H^0(X,L)$, $V_\FF := V \otimes \FF$ and $r := h^0(X,L) - 1$.  Fix a closed immersion 
\begin{equation}\label{complete:linear:system:closed:immersion}
X_\FF \hookrightarrow \PP^r_\FF \text{,}
\end{equation}
which is determined by the complete linear system $|V_\FF|$.  By a \emph{linear section of $X$ with respect to the embedding \eqref{complete:linear:system:closed:immersion} and having field of definition some finite extension $\FF / \KK$ of the base number field $\KK$}, is meant a subscheme of the form $X_{\FF} \bigcap \Lambda$ for $\Lambda$ the $\FF$-span of some linear space $\Lambda \subseteq \PP^r_{\FF}$.
\end{defn}

\begin{remark}
Definition \ref{linear:sections} can be modified to treat the case of big (and not necessarily ample) line bundles $L$ on $X$.  Indeed, by resolving the indeterminacies of the big complete linear system $|L|$ we obtain a concept of linear section with respect to the  base point free linear system on some proper model $X'$ of $X$ which is determined by $L$.  The image in $X$ of such linear sections are defined to be the \emph{linear sections} of $X$ with respect to the big line bundle $L$.  They are contained in the complement $X \setminus \operatorname{Bs}(L)$ for $\operatorname{Bs}(L)$ the \emph{base locus} of $L$.
\end{remark}

Our Second Main Theorem, for points of bounded degree and expressed in terms of linear systems, reads in the following way.  For our purposes here, it suffices to treat the case of big line bundles.  A similar conclusion also remains true for the case that $L$ is only assumed to be effective (compare with \cite{Ru:Vojta:2016} and \cite{Grieve:2018:autissier}).

\begin{theorem}\label{Schlickewei:linear:systems:points:bounded:degree}
The main inequalities for points of bounded degree,  as stated in Conjecture \ref{Points:bounded:degree}, imply the following conclusion.  Let $L$ be a big 
line bundle on a geometrically irreducible projective variety $X$.  Assume that both $X$ and $L$ are defined over a number field $\KK$ and put $r := h^0(X,L) - 1$.  Fix a positive natural number $d$.  There exists a positive constant $\mathrm{c}(d,L)$, which depends only on $d$ and $L$ so that the following holds true.  Let $\FF / \KK$ be a finite extension, $\KK \subseteq \FF \subseteq \overline{\KK}$.  Put $V := \H^0(X,L)$ and $V_\FF := V \otimes_{\KK} \FF$.  Let $S$ be a finite subset of $M_{\KK}$.  For each $v \in S$, fix a collection of linearly independent global sections $s_{v,0},\dots,s_{v,r} \in V_\FF$.  Let $\delta > 0$. 
If $Z \subsetneq X$ is the set of all points $x \in X\left(\overline{\KK}\right)$ which satisfy the conditions that
\begin{itemize}
\item{$
\sum_{v \in S} \sum_{i=0}^r \lambda_{s_{v,i},v}(x) > (\mathrm{c}(d,L) + \delta) h_L(x) + \mathrm{O}(1) $; and }
\item{$
[\KK(x) : \KK] \leq d$,
}
\end{itemize}
then there exist finitely many proper linear sections $\Lambda_1,\dots,\Lambda_h$ in $X_{\overline{\KK}}$ each having field of definition with degree at most $d$ over $\KK$ and such that $Z$ is contained in the union $\Lambda_1 \bigcup \hdots \bigcup \Lambda_h$.
\end{theorem}

\begin{proof}
The reformulation of diophantine arithmetic inequalities for linear forms in polynomial rings into the language of very ample linear systems is now quite standard \cite{Autissier:2011}, \cite{Ru:Vojta:2016}.  Arguing as in \cite[Proposition 4.2]{Autissier:2011},  \cite[Theorem 2.10]{Ru:Vojta:2016} or \cite[Proposition 2.1]{Grieve:2018:autissier}, for example, we conclude that, without loss of generality, we may reduce to the case that $L$ is very ample.  One other key point is that the choice of sections allows for an embedding of $X_\FF = X \times \spec \FF$ into projective $r$ space $\PP^r_{\FF} = \PP^r_{\KK} \times \spec \FF$.   The result then follows by applying the traditional form of the Second Main Theorem, for points of bounded degree, as in Conjecture \ref{Points:bounded:degree} or \cite[Conjecture 5.1]{Schlickewei:2003}, to $(\PP^r_{\FF}, \Osh_{\PP^r_{\FF}}(1))$ combined with the behaviour of local Weil functions under pullback.  Note that it is important that the diophantine exceptional set 
is contained in a finite union of linear spaces.
\end{proof}

\section{Expected orders of vanishing}

The purpose of this section is to define our main measures for expected orders of vanishing.  We work over a fixed base number field $\KK$ and fix a finite extension of number fields $\FF / \KK$, with $\KK \subseteq \FF \subseteq \overline{\KK}$, for $\overline{\KK}$ some fixed algebraic closure of $\KK$.  

Let $L$ be a big line bundle on a geometrically irreducible projective variety $X$.  We  assume that $X$ and $L$ are both defined over $\KK$ and  respectively denote by $X_\FF$ and $L_\FF$ their base change from $\KK$ to $\FF$.  Given a nonzero effective Cartier divisor $D$ on $X$, defined over $\FF$, the main measure of positivity and expected orders of vanishing  for $L$ along $D$ can be defined as
\begin{equation}\label{arithmetic:general:thm:gamma}
\gamma(L,D) := \limsup_{m \to \infty} \frac{m h^0(X,mL) }{ \sum_{\ell \geq 1} h^0(X,mL - \ell D)  } \text{;}
\end{equation}
see \cite{Autissier:2011} and \cite{Ru:Vojta:2016}.

This quantity \eqref{arithmetic:general:thm:gamma} admits several equivalent descriptions (see for example \cite{Grieve:2018:autissier}, \cite{Grieve:toric:gcd:2019}, \cite{Grieve:chow:approx}).  Moreover, for the case that $X$ is a Fano variety, in particular when the anticanonical class $L = - \K_X$ is ample, then the condition that $\gamma(L,D) \leq 1$ is significant from the point of view of $\K$-instability for $(X,-\K_X)$ and Vojta's Main Conjecture for $\KK$-rational points \cite{Grieve:Divisorial:Instab:Vojta}.

Turning to points of bounded degree, by analogy with the inequalities of Theorem \ref{Schlickewei:linear:systems:points:bounded:degree} and their role in obtaining general arithmetic inequalities for points of bounded degree, we define
\begin{equation}\label{arithmetic:general:thm:gamma:bounded:degree}
\gamma(d,L,D) := \limsup_{m \to \infty} \frac{ m \mathrm{c}(d,mL) }{ \sum_{\ell \geq 1} h^0(X, mL - \ell D) } \text{.}
\end{equation}
Here, in \eqref{arithmetic:general:thm:gamma:bounded:degree}, we assume that the constant $\mathrm{c}(d,mL)$, for each $m \geq 0$, is the optimal constant
that is given by Theorem \ref{Schlickewei:linear:systems:points:bounded:degree}.   

We also put
\begin{equation}\label{arithmetic:general:thm:gamma:bounded:degree:eqn:prime}
\mathrm{b}(d,mL) := \frac{\mathrm{c}(d,mL)}{ h^0(X,mL) } \text{.}
\end{equation}
The idea behind making the auxiliary definition \eqref{arithmetic:general:thm:gamma:bounded:degree:eqn:prime} is that, in the limit, we expect that these quantities \eqref{arithmetic:general:thm:gamma} and \eqref{arithmetic:general:thm:gamma:bounded:degree} are related as
$$\gamma(d,L,D) = \mathrm{b}(d) \gamma(L,D)$$
for $\mathrm{b}(d)$
some new constant, which should depend on $X$ and $d$ (but not on $L$ nor $D$).  We do not make further comment about these considerations in what follows.

For use in the proof of Theorem \ref{Arithmetic:General:Theorem:Points:Bounded:Degree}, we make explicit mention of a slight variant to \cite[Proposition 4.18]{Ru:Vojta:2016}.  Our formulation here allows for finite extensions $\FF$ of the base number field $\KK$.  We include a short proof, following closely that of \cite[Proposition 4.18]{Ru:Vojta:2016}, for completeness.  Here, we argue as in \cite[Proposition 4.18]{Ru:Vojta:2016} but pass to the base change from $\KK$ to $\FF$.  Here and elsewhere, we assume some familiarity with the concept of \emph{birational divisor}.  We refer to \cite{Ru:Vojta:2016} for a more detailed discussion, which is suitable for our purposes here, on that topic.

\begin{lemma}[Compare with {\cite[Proposition 4.18]{Ru:Vojta:2016}}]\label{glb:lemma}
Let $L$ be a line bundle on a geometrically irreducible projective variety $X$ and defined over the base number field $\KK$.  Let $\FF / \KK$ be some finite extension field and let $E_1,\dots,E_\ell$ be a collection of effective Cartier divisors on $X$ and defined over $\FF$.  Let
$$
\mathbb{E} = \bigwedge_{i=1}^\ell E_i
$$  
be the greatest lower bound of the $E_j$, for $j = 1,\dots,\ell$, and let 
$$
0 \not = s \in \H^0(X_{\FF},L_{\FF})
$$
be some nonzero global section which lies in
the coherent subsheaf of $L_\FF$ which is generated by the sheaf $\sum_{j=1}^\ell \left(L_{\FF}- E_j\right)$.  Then
\begin{equation}\label{glb:lemma:eqn}
\operatorname{div}(s) \geq \bigwedge_{j=1}^\ell E_j \text{.}
\end{equation}
\end{lemma}

\begin{proof}
Let $\phi \colon X' \rightarrow X_{\FF}$ be some model for which $\mathbb{E}$ has trace a Cartier divisor $E = \mathbb{E}_{X'}$.  Then
$$
\phi^*E_j - E \geq 0
$$
for all $j = 1,\dots, \ell$.  Thus, for each $j = 1,\dots,\ell$, the sheaf $\phi^*(L_\FF - E_j)$ is a subsheaf of $\phi^*L_\FF - E$.  In particular, it follows that
$$
\phi^* s \in \H^0(X',\phi^*L_{\FF} - E) 
$$
whence the desired inequality \eqref{glb:lemma:eqn}.
\end{proof}

\section{Proof of the arithmetic inequalities for points of bounded degree}\label{proof:Arithmetic:General:Theorem}

Here we prove Theorem \ref{Arithmetic:General:Theorem:Points:Bounded:Degree}.  In doing so, we summarize the most important aspects of the filtration method for measuring complexity of rational points.  The method of proof for the arithmetic general theorem, for points of bounded degree, is obtained from Conjecture \ref{Points:bounded:degree}  in a manner which is similar to the classical case.  Recall, that the origin for this method is attributed to Corvaja-Zannier, \cite{Corvaja:Zannier:2002}, and others.  Our conventions about this filtration construction are consistent with those of \cite{Ru:Vojta:2016}.  Note that they are slightly different than those of \cite{Grieve:2018:autissier}.

To prepare for the proof of Theorem \ref{Arithmetic:General:Theorem:Points:Bounded:Degree}, let $D_1,\dots,D_q$ be nonzero effective Cartier divisors intersecting properly on  $X$ and let 
$D = D_1 + \dots + D_q \text{.}$  
Fix a big line bundle $L$ on $X$ and put $n := \dim X$.  We allow that these divisors $D_i$ are defined over $\FF$ a finite extension of the base number field $\KK$ and contained in $\overline{\KK}$ a fixed algebraic closure of $\KK$.  We insist that $X$ is geometrically irreducible and that $L$ is defined over $\KK$.  That the divisors $D_1,\dots,D_q$ intersect properly means that for all subsets $I \subseteq \{1,\dots,q\}$ and all $x \in \bigcap_{i \in I} \operatorname{Supp} D_i$, the local defining equations for the $D_i$ form a regular sequence in the local ring $\Osh_{X_{\FF},x}$. Finally, recall that $\operatorname{Bs}(L)$, the \emph{base locus of $L$}, consists of those points of $X(\overline{\KK})$ at which all global sections of $L$ vanish.

In this context, we proceed with the proof of Theorem \ref{Arithmetic:General:Theorem:Points:Bounded:Degree}.  Again, the proof we give  is an adaptation of the approach given in \cite{Ru:Vojta:2016} and \cite{Grieve:2018:autissier}.  The success of our method is made possible because of Theorem \ref{Schlickewei:linear:systems:points:bounded:degree} and because of the definition given by \eqref{arithmetic:general:thm:gamma:bounded:degree}.

\begin{proof}[Proof of Theorem \ref{Arithmetic:General:Theorem:Points:Bounded:Degree}] 
  Recall, that $n = \dim X$.   Let $X_\FF$ and $L_\FF$ denote, respectively, the base changes of $X$ and $L$ from $\KK$ to $\FF$.  Hencefourth, as no confusion is likely, we denote $X_{\FF}$ and $L_{\FF}$ by $X$ and $L$.

Let $\epsilon > 0$.  Similar to the reduction step described in \cite[p.~985]{Ru:Vojta:2016}, it follows, from the Northcott property for big line bundles (compare with \cite[Theorem 2.4.9]{Bombieri:Gubler}), that there exist constants $A$ and $B$ together with a proper Zariski closed subset $Z \subsetneq X$, which is contained in a finite union of traces of linear subvarieties, each of which has field of definition at most $d$ over $\KK$, such that for all algebraic points
$$
x \in X(\overline{\KK}) \setminus \left( Z(\overline{\KK}) \bigcup \operatorname{Supp}(D_i)(\overline{\KK}) \bigcup \operatorname{Bs}(L)(\overline{\KK})\right)
$$
with 
$$[\KK(x) : \KK] \leq d \text{,}$$ 
if 
$$h_L(x) > B\text{,}$$ then
$$
\sum_{v \in S} \lambda_{D_i,v}(x) < A h_L(x)\text{,}
$$
for all $i = 1,\dots,q$.  Here, $\operatorname{Bs}(L)(\overline{\KK})$ denotes the base locus of $L$.  

In light of this, to prove the theorem, we may and do replace the given $\epsilon > 0$ by some smaller $\epsilon > 0$ and each of the quantities $\gamma(d,L,D_i)$, for $i = 1,\dots,q$, by some slightly larger rational numbers $\gamma_i$.  Put 
$$\beta_i := \gamma_i^{-1}.$$  
Then
$$
\beta_i < \beta(d, L, D_i) := \gamma(d,L,D_i)^{-1}
$$
for all $i = 1,\dots,q$.  Here and elsewhere, we have defined
$$
\gamma(d,L,D_i) := \limsup_{m \to \infty} \frac{ m \mathrm{c}(d,mL) }{ \sum_{\ell \geq 1} h^0(X,mL-\ell D_i) } \text{.}
$$

Having fixed such $\epsilon > 0$ and such rational numbers $\beta_i$, for all $i = 1,\dots,q$, our proof of Theorem \ref{Arithmetic:General:Theorem:Points:Bounded:Degree} is achieved by modifying the filtration construction of \cite{Ru:Vojta:2016}.  Again, our conventions in regards to this construction follows those of \cite{Ru:Vojta:2016} closely.  They differ slightly from those our previous work \cite{Grieve:2018:autissier}.  

In particular, for our purposes here, choose $\epsilon_1 > 0$ and positive integers $m$ and $b$ so that 
\begin{equation}\label{key:eqn1}
\left(1 + \frac{n}{b} \right) \max_{1 \leq i \leq q} \frac{ \beta_i m \mathrm{c}(d,mL) + m \epsilon_1}{\sum_{\ell \geq 1} h^0(X,m L -\ell D_i)} <  1 + \epsilon.
\end{equation}

Having now fixed $m$ and $b$, let
$$
\Sigma := \left\{ \sigma \subseteq \{1,\dots,q\} : \bigcap_{j \in \sigma} \operatorname{Supp} (D_j) \not = \emptyset  \right\} \text{.}
$$
For $\sigma \in \Sigma$, let
$$
\Delta_{\sigma} := \left\{ 
\mathbf{a} = (a_i) \in \prod_{i \in \sigma} \gamma_i \NN  : \sum_{i \in \sigma} \beta_i a_i = b 
\right\}
$$
and for each 
$\mathbf{a} \in \Delta_{\sigma}\text{,}$ 
define the ideal sheaf 
$$\Ish_{\mathbf{a}}(t) = \Ish(t) \subseteq \Osh_X$$ 
by the rule
$$
\Ish(t) := \sum_{\mathbf{b}} \Osh_X \left(-\sum_{i \in \sigma} b_i D_i \right) \text{.}
$$
Here, the sum is taken over all 
$\mathbf{b} \in \NN^{\# \sigma}$ 
which have the property that  
$$
\sum_{i \in \sigma} a_i b_i \geq t \text{.}
$$
Put 
$$
\mathcal{F}(\sigma;\mathbf{a})_t = \H^0(X,L^{\otimes m} \otimes \Ish(t)) \subseteq \H^0(X, mL).
$$

Now, for each nonzero section 
$$
0 \not = s \in \H^0(X, m L)
$$
define
$$
\mu_{\mathbf{a}}(s) = \sup \left\{t \in \RR_{\geq 0} : s \in \mathcal{F}(\sigma;\mathbf{a})_t  \right\} 
$$
and let $\mathcal{B}_{\sigma;\mathbf{a}}$ be a basis of $\H^0(X, m L)$ which is adapted to the filtration 
$$
\{\mathcal{F}(\sigma;\mathbf{a})_t \}_{t \in \RR_{\geq 0} } \text{.}
$$ 
In particular, by this is meant that
$$\mathcal{B}_{\sigma;\mathbf{a}} \bigcap \mathcal{F}(\sigma;\mathbf{a})_t$$ 
is a basis for $\mathcal{F}(\sigma;\mathbf{a})_t$ for all $t \in \RR_{\geq 0}$.

Set
\begin{equation}\label{filtration:expectation:Eqn}
F(\sigma,\mathbf{a}) := \frac{1}{h^0(X, mL)} \sum_{s \in \mathcal{B}_{\sigma; \mathbf{a}}} \mu_{\mathbf{a}}(s) \text{.}
\end{equation}
A key point is then to observe that
\begin{equation}\label{key:eqn:5:0}
F(\sigma;\mathbf{a})  \geq  
\min_{1 \leq i \leq q} \left( \frac{b  }{ \beta_i h^0(X, m L)} \sum_{\ell \geq 1} h^0(X, mL - \ell D_i)\right) \text{.}
\end{equation}
Here, we have used 
the fact that the divisors $D_1,\dots, D_q$ intersect properly. 
(See \cite[Proposition 6.7]{Ru:Vojta:2016} or \cite[Theorem 3.6]{Autissier:2011} for further details in regards to this second assertion.)

Hence, upon combining \eqref{filtration:expectation:Eqn} and \eqref{key:eqn:5:0}
\begin{equation}\label{eqn:37}
\sum_{s \in \mathcal{B}_{\sigma;\mathbf{a}}} \mu_{\mathbf{a}}(s) \geq \min_{1 \leq i \leq q} \left( \frac{b}{\beta_i} \sum_{\ell \geq 1} h^0(X, m L - \ell D_i)\right) \text{.}
\end{equation}
Note also that
\begin{equation}\label{eqn:5.2}
L^{\otimes m} \otimes \Ish(\mu_{\mathbf{a}}(s)) 
= 
\sum_{\mathbf{b} \in K} 
\left(
m L -\sum_{i \in \sigma} b_i D_i 
\right)
\end{equation}
where 
\begin{equation}\label{eqn:5.3}
K = K_{\sigma,\mathbf{a},s}
\end{equation}
is the set of minimal elements of the set
\begin{equation}\label{eqn:5.4}
\left \{ \mathbf{b} \in \NN^{\# \sigma} : \sum_{i \in \sigma} a_i b_i \geq  \mu_{\mathbf{a}}(s) \right \}
\end{equation}
relative to the product partial ordering on $\NN^{\# \sigma}$.

Working over some normal proper model of $X$, two other useful facts, for our purposes here, are the inequalities 
\begin{equation}\label{greatest:lower:bound}
\operatorname{div}(s) \geq \bigwedge_{\mathbf{b} \in K} \sum_{i \in \sigma} b_i D_i
\end{equation}
and
\begin{equation}\label{least:upper:bound}
\bigvee_{\substack{
\sigma \in \Sigma \\
\mathbf{a} \in \Delta_{\sigma}
}
} \operatorname{div} (\mathcal{B}_{\sigma;\mathbf{a}}) \geq \frac{b}{b+n} \left(\min_{1 \leq i \leq q} \sum_{\ell =1}^\infty \frac{h^0(X, m L - \ell D_i )}{ \beta_i } \right) \sum_{i = 1}^q \beta_i D_i \text{,}
\end{equation}
which were noted in \cite[Equations (41) and (42)]{Ru:Vojta:2016}.  (In equations \eqref{greatest:lower:bound} and \eqref{least:upper:bound}, the notations $\bigwedge$ and $\bigvee$ denote, respectively, the greatest lower bound and least upper bound.) 

Here, for the sake of completeness, we explain, following \cite[Proof of Lemma 6.8]{Ru:Vojta:2016}, the manner in which equation \eqref{greatest:lower:bound} can be used to establish equation \eqref{least:upper:bound}.

First, we mention that equation \eqref{greatest:lower:bound} is implied by Lemma \ref{glb:lemma}.  Indeed, because of \eqref{eqn:5.2}, combined with finiteness of the set \eqref{eqn:5.3}, the desired inequality \eqref{greatest:lower:bound} is achieved as a direct application of Lemma \ref{glb:lemma}.

Next, let us turn our attention to establishing the inequality \eqref{least:upper:bound}.  With this aim in mind, let
$$
\mathbb{D} := \bigvee_{\substack{ \sigma \in \Sigma 
\\ 
\mathbf{a} \in \Delta_{\sigma} }
}
\operatorname{div}(\mathcal{B}_{\sigma; \mathbf{a}})
$$
be the least upper bound of the divisors
$$
\operatorname{div}(\mathcal{B}_{\sigma;\mathbf{a}}) := 
\sum_{
s' \in \mathcal{B}_{\sigma;\mathbf{a}}
} 
\operatorname{div}(s') \text{.}
$$

Fix a normal projective model 
$$\phi \colon X' \rightarrow X$$ 
on which $\mathbb{D}$ is represented by a Cartier divisor 
$$D' = \mathbb{D}_{X'}\text{;}$$ 
let $E$ be a prime Weil divisor on $X'$.  Fix an arbitrary point 
$$x \in \phi(\operatorname{Supp}(E))$$ 
and put
$$
\sigma := \{i \in \{1,\dots,q\} : x \in \operatorname{Supp}(D_i) \} \text{.}
$$

For all 
$\mathbf{a} \in \Delta_{\sigma} \text{,}$
let $\nu',\nu_{\sigma,\mathbf{a}}$ and $\nu_i$, for $i = 1,\dots,q$, be the multiplicities of $E$ in $D'$, $\phi^* \operatorname{div}(\mathcal{B}_{\sigma,\mathbf{a}})$ and $D_i$, respectively.    
In particular, note that if 
$$\mathbf{a} \in \Delta_{\sigma}\text{,}$$ 
then
$$
\nu' \geq \nu_{\sigma,\mathbf{a}}.
$$
Put 
$$\nu = \sum_{i=1}^q \beta_i \nu_i\text{.}$$

The desired inequality \eqref{least:upper:bound} is then achieved after establishing existence of some 
$$\mathbf{a} \in \Delta_{\sigma}$$ 
which has the property that
\begin{equation}\label{eqn:40}
\nu_{\sigma,\mathbf{a}} \geq \frac{b}{b+n} \left( \min_{1 \leq i \leq q} \sum_{\ell = 1}^\infty \frac{h^0(X, m L - \ell D_i) }{\beta_i} \right) \nu \text{.}
\end{equation}
To establish the inequality \eqref{eqn:40}, observe first that if $\nu = 0$, then there is nothing to prove.  On the other hand, suppose that $\nu > 0$.  For all $i \in \sigma$, put
\begin{equation}\label{eqn:41}
t_i := \nu_i / \nu \text{.}
\end{equation}
Then, by construction, $\nu_i = 0$ for all $i \not \in \sigma$.  Thus
$$
\sum_{i \in \sigma} \beta_i \nu_i = \sum_{i=1}^q \beta_i \nu_i = \nu;
$$
whence
$$
\sum_{i \in \sigma} \beta_i t_i = 1.
$$

Now, by assumption, the divisors $D_1,\dots,D_q$ intersect properly.  In particular, they are in general position.  Among other things, it follows that
$$
\# \sigma \leq  n = \dim X.
$$
Furthermore 
$$
b \leq \sum_{i \in \sigma} \lfloor (b+n)\beta_i t_i \rfloor \leq b + n.
$$
Fix
$$
\mathbf{a} = (a_i) \in \Delta_{\sigma}
$$
which has the property that
\begin{equation}\label{eqn:42}
t_i \geq \frac{a_i}{b+n}
\end{equation}
for all $i \in \sigma$.

For each $s \in \mathcal{B}_{\sigma;\mathbf{a}}$, let $\nu_s$ be defined as 
$$
\nu_s := \operatorname{mult}_{\phi^*(s)} (E);
$$
in other words, $\nu_s$ is the multiplicity of $E$ in the divisor $\operatorname{div}(\phi^*(s))$.  

The above discussion implies, especially because of \eqref{greatest:lower:bound}, \eqref{eqn:41} and \eqref{eqn:42}, combined with the fact that for all $\mathbf{b} \in K$,
$$
\sum_{i \in \sigma} a_i b_i \geq  \mu_{\mathbf{a}}(s) \text{,}
$$
the following sequence of inequalities
\begin{align}\label{eqn:43}
\begin{split}
\nu_s & \geq \min_{\mathbf{b} \in K} \sum_{i \in \sigma} b_i \nu_i \cr
& = \left( \min_{\mathbf{b} \in K} \sum_{i \in \sigma} b_i t_i \right) \nu \cr
& \geq \left( \min_{\mathbf{b} \in K} \sum_{i \in \sigma} \frac{a_i b_i}{b+n} \right) \nu \cr
& \geq \left( \frac{1}{b+n} \right) \mu_{\mathbf{a}} (s) \nu  \text{.} 
\end{split}
\end{align}
In \eqref{eqn:43}, the set $K$ is as in \eqref{eqn:5.3}, namely, it is the minimal elements of the set \eqref{eqn:5.4}.

But now, combining the relations  \eqref{eqn:43} and \eqref{eqn:37}, it then follows that 
\begin{align}\label{eqn:44}
\begin{split}
\frac{\nu_{\sigma,\mathbf{a}}}{\nu} &
= \frac{1}{\nu} \left( \sum_{s \in \mathcal{B}_{\sigma;\mathbf{a}}} \nu_s \right) \cr
& \geq \frac{1}{b+n} \left( \sum_{s \in \mathcal{B}_{\sigma;\mathbf{a}}} \mu_{\mathbf{a}}(s) \right) \cr
& \geq \frac{b}{b+n} \left(  \min_{1 \leq i \leq q} \sum_{\ell \geq 1} \frac{h^0(X, mL - \ell D_i)}{\beta_i}  \right) \text{.}
\end{split}
\end{align} 
This last collection of inequalities, given in \eqref{eqn:44}, establish the inequality \eqref{eqn:40} and thus the inequality \eqref{least:upper:bound}.

Now, having established \eqref{least:upper:bound} and \eqref{greatest:lower:bound}, 
write
$$
\bigcup_{\sigma ; \mathbf{a}} \mathcal{B}_{\sigma;\mathbf{a}} = \mathcal{B}_1 \bigcup \hdots \bigcup \mathcal{B}_{T_1} = \{s_1,\dots,s_{T_2}\}.
$$
For each $i = 1, \dots, T_1$, let 
$$J_i \subseteq \{1,\dots,T_2\}$$ 
be the subset such that
$$\mathcal{B}_i = \{s_j : j \in J_i \} \text{.}$$  
Choose Weil functions, for each $v \in S$, $\lambda_{D_i,v} \left( \cdot \right)$, $\lambda_{\operatorname{div}\left(\mathcal{B}_i \right), v}\left( \cdot \right)$, for $i = 1,\dots, T_1$ and $\lambda_{s_j, v}\left( \cdot \right)$, for $j = 1,\dots, T_2$, for the divisors $D_i, \operatorname{div}(\mathcal{B}_i)$ and $s_j$ respectively.

Then for each $v \in S$, it holds true, using \eqref{least:upper:bound}, that
\begin{align}\label{key:eqn2}
\begin{split}
\frac{b}{b+n}\left( \min_{1 \leq i \leq q} \sum_{\ell \geq 1} \frac{h^0(X, m L - \ell D_i)}{\beta_i}\right) \sum_{i = 1}^q \beta_i \lambda_{D_i,v} \left( \cdot \right)   &  \leq 
 \max_{1 \leq i \leq T_1} \lambda_{\operatorname{div} \left( \mathcal{B}_i \right),v}(\cdot) + \mathrm{O}_v(1)
 \\ 
  & = \max_{1 \leq i \leq T_1} \sum_{j \in J_i} \lambda_{s_j,v}(\cdot) + \mathrm{O}_v(1).
  \end{split}
\end{align}

With $\epsilon_1$ in place of $\delta$, Theorem \ref{Schlickewei:linear:systems:points:bounded:degree} implies existence of a proper Zariski closed subset $Z \subsetneq X$, which is contained in a finite union of linear sections $\Lambda_1,\dots,\Lambda_h$ each of which has field of definition with degree at most equal to $d$ over $\KK$, so that if
$$x \in X(\overline{\KK}) \setminus \left( Z(\overline{\KK}) \bigcup \operatorname{Bs}(L)(\overline{\KK}) \right)$$ 
and 
$$[\KK(x):\KK]\leq d\text{,}$$ 
then
\begin{equation}\label{key:eqn3}
\sum_{v \in S} \max_J \sum_{j \in J} \lambda_{s_j,v}(x) \leq \left(\mathrm{c}(d, mL) + \epsilon_1 \right) h_{m L}(x) + \mathrm{O}(1).
\end{equation}
(The maximum is taken over all subsets 
$$J \subseteq \{1,\dots,T_2\}$$ 
for which the sections $s_j$, for $j \in J$, are linearly independent.)

It follows, upon combining Equations \eqref{key:eqn2} and \eqref{key:eqn3}, that
\begin{equation}\label{key:eqn4}
\sum_{i=1}^q \beta_i m_S(x,D_i)  \leq \left( 1 + \frac{n}{b} \right) \max_{1 \leq i \leq q} \left(  \frac{\beta_i(\mathrm{c}(d, m L) + \epsilon_1) }{ \sum_{\ell \geq 1} h^0(X, m L - \ell D_i )}\right) h_{mL}(x)+ \mathrm{O}(1)
\end{equation}
for all 
$$x \in X(\overline{\KK}) \setminus \left( Z(\overline{\KK}) \bigcup \operatorname{Bs}(L)(\overline{\KK}) \bigcup \operatorname{Supp}(D)(\overline{\KK}) \right)$$ 
with 
$$[\KK(x) : \KK] \leq d \text{.}$$  
(Here, we note that all $J_i$ occur among the $J$ in the above.)  
Finally, since 
$$
h_{m L}(x) = m h_{L}(x),$$ 
it holds true, by our choice of $\beta_i$ and by combining Equations \eqref{key:eqn1} and \eqref{key:eqn4}, that 
$$
\sum_{i=1}^q \gamma(d,L,D_i)^{-1} m_S(x,D_i) \leq \left(1 + \epsilon \right) h_{L}(x) + \mathrm{O}(1) 
$$
for all points
$$x \in X(\overline{\KK}) \setminus \left( Z(\overline{\KK}) \bigcup \operatorname{Bs}(L)(\overline{\KK}) \bigcup \operatorname{Supp}(D)(\overline{\KK})\right)$$ 
which have the property that 
$$[\KK(x):\KK] \leq d\text{.}$$
This completes our proof of Theorem \ref{Arithmetic:General:Theorem:Points:Bounded:Degree}.
\end{proof}

\ssubsection{Acknowledgements}
This work began while I was a postdoctoral fellow at Michigan State University.  It 
benefited from visits to CIRGET, Montreal, ICERM, Providence, and to the Institute of Mathematics Academia Sinica, Taipei, during the Summer of 2019.  It is my pleasure to thank Aaron Levin, Steven Lu, Mike Roth, Min Ru, Julie Wang and many other colleagues for their interest and conversations on related topics.  Finally, I thank anonymous referees for carefully reading earlier versions of this article and for providing comments and suggestions.

\providecommand{\bysame}{\leavevmode\hbox to3em{\hrulefill}\thinspace}
\providecommand{\MR}{\relax\ifhmode\unskip\space\fi MR }
\providecommand{\MRhref}[2]{%
  \href{http://www.ams.org/mathscinet-getitem?mr=#1}{#2}
}
\providecommand{\href}[2]{#2}


\begin{thebibliography}{10}

\bibitem{Autissier:2011}
P.~Autissier, \emph{Sur la non-densit\'e des points entiers}, Duke Math. J.
  \textbf{158} (2011), no.~1, 13--27.

\bibitem{Bombieri:Gubler}
E.~Bombieri and W.~Gubler, \emph{Heights in {D}iophantine geometry}, Cambridge
  University Press, Cambridge, 2006.

\bibitem{Corvaja:Zannier:2002}
P.~Corvaja and U.~Zannier, \emph{A subspace theorem approach to integral points
  on curves}, C. R. Math. Acad. Sci. Paris \textbf{334} (2002), no.~4,
  267--271.

\bibitem{Evertse:Ferretti:2013}
J.-H. Evertse and R.~G. Ferretti, \emph{A further improvement of the
  {Q}uantitative {S}ubspace {T}heorem}, Ann. of Math. (2) \textbf{177} (2013),
  513--590.

\bibitem{Evertse:Schlickewei:2002}
J.-H. Evertse and H.~P. Schlickewei, \emph{A quantitative version of the
  absolute subspace theorem}, J. Reine Angew. Math. \textbf{548} (2002),
  21--127.

\bibitem{Faltings:Wustholz}
G.~Faltings and G.~W\"{u}stholz, \emph{Diophantine approximations on projective
  spaces}, Invent. Math. \textbf{116} (1994), 109--138.

\bibitem{Grieve:Function:Fields}
N.~Grieve, \emph{Diophantine approximation constants for varieties over
  function fields}, Michigan Math. J. \textbf{67} (2018), no.~2, 371--404.

\bibitem{Grieve:2018:autissier}
\bysame, \emph{On arithmetic general theorems for polarized varieties}, Houston
  J. Math. \textbf{44} (2018), no.~4, 1181--1202.
  
\bibitem{Grieve:toric:gcd:2019}
\bysame, \emph{Generalized {GCD} for {T}oric {F}ano {V}arieties}, Acta Arith.
  \textbf{195} (2020), no.~4, 415--428.
  
\bibitem{Grieve:chow:approx}
\bysame, \emph{Expectations, concave transforms, {C}how weights, and {R}oth's
  theorem for varieties}, Preprint.

\bibitem{Grieve:Divisorial:Instab:Vojta}
\bysame, \emph{Divisorial instability and {V}ojta's main conjecture for
  {$\mathbb{Q}$}-{F}ano varieties}, Asian J. Math. (To appear).

\bibitem{Levin:2014}
A.~Levin, \emph{On the {S}chmidt subspace theorem for algebraic points}, Duke
  Math. J. \textbf{163} (2014), no.~15, 2841--2885.

\bibitem{Ru:Vojta:2016}
M.~Ru and P.~Vojta, \emph{A birational {N}evanlinna constant and its
  consequences}, Amer. J. Math. \textbf{142} (2020), no.~3, 957--991.

\bibitem{Schlickewei:2003}
H.~P. Schlickewei, \emph{Approximation of algebraic numbers}, Diophantine
  approximation ({C}etraro, 2000), Lecture Notes in Math., vol. 1819, Springer,
  Berlin, 2003, pp.~107--170.

\bibitem{Vojta:1998}
P.~Vojta, \emph{A more general {$abc$} conjecture}, Internat. Math. Res.
  Notices (1998), no.~21, 1103--1116.

\end{thebibliography}
\end{document}